\documentclass[twoside,a4paper]{amsart}

\usepackage[english]{babel}
\usepackage{amssymb}
\usepackage{dsfont}
\usepackage{amsmath}

\newtheorem{theorem}{Theorem}[section]
\newtheorem{corollary}[theorem]{Corollary}
\newtheorem{lemma}[theorem]{Lemma}
\newtheorem{definition}{Definition}[section]
\newtheorem{proposition}[theorem]{Proposition}
\newtheorem{remark}{Remark}[section]

\usepackage{color}

\usepackage{amssymb,amscd,amsthm,latexsym}
\usepackage[all]{xy}

\newdir{ >}{
  @{}*!/-10pt/@{>} }

\newdir{ |>}{
  @{}*!/-5.5pt/@{|}*!/-10pt/:(1,-.2)@^{>}*!/-10pt/:(1,+.2)@_{>} }

\newcommand{\CC}{ \ensuremath{\mathbb {C}} }
\newcommand{\V}{\mathcal V}
\newcommand{\XX}{ \ensuremath{\mathbb {X}} }

\author{Marino Gran}
\address{Institut de Recherche en Math\'ematique et Physique, Universit\'e Catholique de Louvain, Chemin du Cyclotron 2,
1348 Louvain-la-Neuve, Belgium} 
\email{marino.gran@uclouvain.be}

\author{Diana Rodelo}
\address{CMUC, Universidade de Coimbra,
3001--501 Coimbra, Portugal\newline Departamento de Matem\'atica,
Faculdade de Ci\^{e}ncias e Tecnologia, Universidade do Algarve,
Campus de Gambelas, 8005--139 Faro, Portugal}
\email{drodelo@ualg.pt}

\title{Some remarks on pullbacks in Gumm categories}

\dedicatory{Dedicated to Manuela Sobral on the occasion of her seventieth birthday}



\thanks{The second author was supported by the Centro de Matem\'{a}tica da
Universidade de Coimbra (CMUC), funded by the European Regional
Development Fund through the program COMPETE and by the Portuguese
Government through the FCT - Funda\c{c}\~{a}o para a Ci\^{e}ncia e
a Tecnologia under the projects PEst-C/MAT/UI0324/2013 and
PTDC/MAT/120222/2010 and grant number SFRH/BPD/69661/2010.}

\begin{document}

\begin{abstract}
We extend some properties of pullbacks which are known to hold in a Mal'tsev context to the more general context of Gumm categories. The varieties of universal algebras which are Gumm categories are precisely the congruence modular ones. These properties lead to a simple alternative proof of the known property that central extensions and normal extensions coincide for any Galois structure associated with a Birkhoff subcategory of an exact Goursat category.
\end{abstract}

\maketitle


\section{Introduction}\label{Introduction}
\setcounter{equation}{0}
A categorical approach to the property of congruence modularity, well known in universal algebra, was proposed in \cite{BG0,BG1} via a categorical formulation of the so-called \emph{Shifting Lemma} (recalled in Section \ref{Congruence modularity}).
The categories satisfying this categorical property are called \emph{Gumm categories}, since it was the mathematician H.P. Gumm who proved that, for a variety of universal algebras, the validity of the Shifting Lemma is equivalent to congruence modularity \cite{Gumm}. As examples of Gumm categories, we also have regular Mal'tsev categories~\cite{CLP,CKP} and regular Goursat categories~\cite{CKP}, which are defined by the property that any pair of equivalence relations $R$ and $S$ in $\CC$ (on a same object) $3$-permute: $RSR=SRS$.

In this context \cite{BournFibrations} D. Bourn established an interesting permutability result (see Theorem~\ref{Bourn}), that we use in the present paper to prove the following property of regular Gumm categories (Proposition \ref{regular Gumm}). Given a commutative diagram
$$
\label{split |1|2|}
  \vcenter{\xymatrix@!0@R=45pt@C=60pt{
    Z\times_V U \ar@<-2pt>@{>>}[d]_-{} \ar@{>>}[r]  \ar@{}[dr]|-{1} &
    X \ar@<-2pt>@{>>}[d] \ar[r] \ar@{}[dr]|-{2} & U \ar[d] \\
    Z \ar@{>>}[r] \ar@<-2pt>[u] & Y \ar[r] \ar@<-2pt>[u] & V }}
$$
in a regular Gumm category such that the whole rectangle is a pullback and the left square {\footnotesize\fbox{1}} is composed by vertical split epimorphisms and horizontal regular epimorphisms, then both squares {\footnotesize\fbox{1}} and  {\footnotesize\fbox{2}} are pullbacks. This property is known to hold in any regular Mal'tsev category, and has been used, for example, in the categorical theory of central extensions \cite{Gran, BR}.

In the present article we also show that this property can be used to give a new proof of a remarkable property of exact Goursat categories, namely the fact that central extensions and normal extensions relative to any (admissible) Birkhoff subcategory  $\XX$ of $\CC$ coincide \cite{JK}.
Let us recall that a full reflective subcategory  $\XX$ of an exact category $\CC$ is called a Birkhoff subcategory when $\XX$ is closed in $\CC$ under subobjects and regular quotients. In particular, a Birkhoff subcategory of a variety of universal algebras is just a subvariety. A Birkhoff subcategory $\XX$ is admissible, from the point of view of Categorical Galois Theory, when the reflector $I\colon \XX \rightarrow \CC$ preserves pullbacks of regular epimorphisms in $\XX$ along any morphism in $\CC$.
The notions of central extension and of normal extension are defined relatively to the choice of the admissible Birkhoff subcategory $\XX$ of $\CC$, as recalled in Section \ref{Galois}. It is precisely the useful property of pullbacks in regular Gumm categories stated above which allows one to find a simple proof of the coincidence of these two notions in the exact Goursat context (Theorem \ref{central split} and Corollary \ref{Coincidence}).

In~\cite{JK} G. Janelidze and G.M. Kelly proved that every Birkhoff subcategory $\XX$ of an exact category $\CC$ with modular lattice of equivalence relations (on any object in $\CC$) is always admissible. It was later shown by V. Rossi in ~\cite{Rossi} that the same admissibility property still holds in the more general context of Gumm categories which are \emph{almost exact}, a notion introduced by G. Janelidze and M. Sobral in \cite{JS}. We conclude the article by relating our observations on Gumm categories with these results concerning the admissibility of Galois structures.

\section{Preliminaries}\label{Calculus of relations}
In the present paper the term \emph{regular category}~\cite{EC} will be used for a finitely complete category such that any kernel pair has a coequaliser and, moreover, regular epimorphisms are stable under pullbacks. Any morphism $f \colon A \rightarrow B$ in a regular category $\CC$ has a factorisation $f=m\cdot p$, with $p$ a regular epimorphism and $m$ a monomorphism. It is well known that such factorisations are necessarily stable under pullbacks in a regular category. The subobject determined by the monomorphism $m$ in the factorisation of $f=m\cdot p$ is usually called the \emph{image} of $f$.

A relation $R$ from $A$ to $B$ is a subobject $\langle r_1,r_2 \rangle\colon R\rightarrowtail A\times B$. The opposite relation, denoted by $R^{\circ}$, is the relation from $B$ to $A$ given by the subobject $\langle r_2,r_1 \rangle\colon R\rightarrowtail B\times A$. A morphism $f:A\rightarrow B$ can be identified with the relation $\langle 1_A,f \rangle\colon A\rightarrowtail A\times B$; its opposite relation $\langle f, 1_A \rangle\colon A\rightarrowtail B\times A$ is usually referred to as $f^{\circ}$. Given another relation $\langle s_1,s_2 \rangle\colon S\rightarrowtail B\times C$ from $B$ to $C$, the composite relation of $R$ and $S$ is a relation $SR$ from $A$ to $C$, obtained as the image of the induced arrow  $\langle  r_1 \pi_1, s_2  \pi_2  \rangle\colon R \times_B S \rightarrowtail A \times C$, where $(R\times_B S, \pi_1, \pi_2)$ is the pullback of $r_2$ and $s_1$. With the above notation, we can write any relation $\langle r_1,r_2 \rangle\colon R\rightarrowtail A\times B$ as $R=r_2 r_1^{\circ}$. The following properties are well known, and also easy to check (see \cite{CKP}, for instance):

\begin{lemma}
\label{pps of ms as relations}
Let $f: A\rightarrow B$ be any morphism in a regular category $\CC$. Then:
\begin{enumerate}
 \item {$f f^{\circ} f=f$ and $f^{\circ} f f^{\circ}=f^{\circ}$;}
  \item{
  $f f^{\circ}=1_B$ {if and only if} $f$
  is a regular {epimorphism}.}
\end{enumerate}
\end{lemma}

A relation $R$ from an object $A$ to $A$ is called a \emph{relation on $A$}. Such a relation is \emph{reflexive} if $1_A\leqslant R$, \emph{symmetric} if $R^{\circ} \leqslant R$, and \emph{transitive} when
$RR \leqslant R$. A relation $R$ on $A$ is called an \emph{equivalence relation}  when it is reflexive, symmetric and transitive. Any kernel pair $\langle f_1,f_2 \rangle\colon R_f \rightarrowtail A\times A$ of a morphism $f: A\rightarrow B$ is an equivalence relation, called an \emph{effective} equivalence relation. By using the composition of relations, it can be written either as $R_f=f^{\circ}f$, or as $R_f=f_2f_1^{\circ}$. Of course, if $f=m\cdot p$ is the (regular epimorphism, monomorphism) factorisation of an arbitrary morphism $f$, then $R_f$=$R_p$, so that an effective equivalence relation is always the kernel pair of a regular epimorphism.

\section{Gumm categories}\label{Congruence modularity}
A lattice $(L, \vee, \wedge)$ is called \emph{modular} when, for $x,y,z\in L$, one has
$$
    x\leqslant z \Rightarrow x\vee (y\wedge z)=(x\vee y)\wedge z.
$$
A variety $\V$ of universal algebras is called \emph{congruence modular} when every lattice of congruences (= effective equivalence relations) on any algebra in $\V$ is modular. It is well known from ~\cite{Gumm} that a variety $\V$ is congruence modular if and only if the following property, called the Shifting Lemma, holds in $\V$: \\

\noindent \textbf{Shifting Lemma} \\
Given congruences $R,S$ and $T$ on the same algebra $X$ such that $R\wedge S\leqslant T$, whenever $x,y,t,z$ are elements in $X$ with $(x,y)\in R \wedge T$, $(x,t) \in S$, $(y,z)\in S$ and $(t,z)\in R$, it then follows that $(t,z) \in T$:
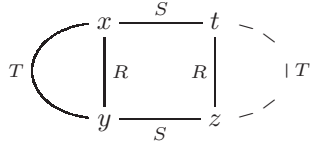
$$
  \xymatrix@C=30pt{
    x \ar@{-}[r]^-S \ar@{-}[d]^-R \ar@(l,l)@{-}[d]_-T & t \ar@{-}[d]_-R \ar@(r,r)@{--}[d]^-T \\
    y \ar@{-}[r]_-S  & z }
$$

This notion has been extended to a categorical context in \cite{BG1}. Indeed, the property expressed by the Shifting Lemma can be equivalently reformulated in any finitely complete category $\mathbb C$ by asking that a specific class of internal functors are discrete fibrations, as we are now going to recall. For any object $X$ in $\mathbb C$ and any equivalence relations $R$, $S$, $T$ on $X$ with $$R \wedge S \le T \le R$$ there is a canonical
inclusion $(i,j) \colon T \square S \rightarrow R \square S$ of equivalence relations, depicted as
\begin{equation}\label{discrete}
 \vcenter{\xymatrix@=50pt{ {T \square S\, } \ar@{>->}[r]^{j}
\ar@<1ex>[d]^{\pi_{2}} \ar@<-1ex>[d]_{\pi_{1}}
 & {R \square S}  \ar@<1ex>[d]^{\pi_{2}} \ar@<-1ex>[d]_{\pi_{1}}
 & \\ {T \,} \ar@{>->}[r]_{i}  & R, &
 }}
\end{equation}
where $T \square S$ (respectively, $R \square S$) is the largest double equivalence relation on $T$ and $S$ (respectively, on $R$ and $S$) and $\pi_1$ and $\pi_2$ are the projections on $T$ (respectively, on $R$).
\begin{definition}\cite{BG1}
A finitely complete category $\CC$ is called a \emph{Gumm category} when any inclusion $(i,j) \colon T \square S \rightarrow R \square S$ as in (\ref{discrete}) is a discrete fibration. This means that any of the commutative squares in (\ref{discrete}) is a pullback.
\end{definition}

Let us recall that a \emph{Mal'tsev} category $\CC$ is a finitely complete category such that every reflexive relation in $\CC$ is an equivalence relation. A regular category $\CC$ is a Mal'tsev category when the composition of (effective) equivalence relations on any object in $\CC$ is $2$-permutable: $R S=S R$, where $R$ and $S$ are (effective) equivalence relations on a same object (see \cite{CLP, CKP}). The strictly weaker $3$-permutability property for (effective) equivalence relations, $RSR=SRS$, defines the notion of regular \emph{Goursat categories}~\cite{CKP}. Goursat categories, thus in particular Mal'tsev categories, have the property that every lattice of equivalence relations (on the same object) is modular (Proposition 3.2 in~\cite{CKP}). This fact implies that any regular Mal'tsev category and, more generally, any regular Goursat category is a Gumm category. Thanks to the characterization theorem in \cite{Gumm} one knows that a variety of universal algebras is congruence modular if and only if it is a Gumm category.

The following property of regular Gumm categories, due to D. Bourn, will play a fundamental role in the next section.

\begin{theorem}\emph{(Theorem 7.12 of~\cite{BournFibrations})}\label{Bourn} Let $\CC$ be a regular Gumm category. Consider equivalence relations $R,S$ and $T$ on a same object such that $RS=SR$ and $R\wedge S\leqslant T \leqslant R$. Then $TS=ST$.
\end{theorem}

\section{Pullback properties in regular Gumm categories}
\label{Pullback properties}
In this section we extend some useful properties of pullbacks from the context of Mal'tsev categories to that of Gumm categories.

The first observation concerns a generalisation of Proposition 3.4 in \cite{BR}:
\begin{proposition}\label{regular Gumm}
Let $\CC$ be a regular Gumm category. Consider a commutative diagram
\begin{equation}
\label{split |1|2|}
  \vcenter{\xymatrix@!0@R=45pt@C=60pt{
    Z\times_V U \ar@<-2pt>@{>>}[d]_-{\varphi} \ar@{>>}[r]^-x  \ar@{}[dr]|-{1} &
    X \ar@<-2pt>@{>>}[d]_-f \ar[r]^-u \ar@{}[dr]|-{2} & U \ar[d]^-{w}  \\
    Z \ar@{>>}[r]_-y \ar@<-2pt>[u]_-{\sigma} & Y \ar[r]_-v \ar@<-2pt>[u]_-s & V, }}
\end{equation}
such that the whole rectangle is a pullback and the left square \emph{\footnotesize\fbox{1}} is composed by vertical split epimorphisms and horizontal regular epimorphisms. Then both squares \emph{\footnotesize\fbox{1}} and \emph{\footnotesize\fbox{2}} are pullbacks.
\end{proposition}
\begin{proof}
The comparison morphism $\langle \varphi, x \rangle\colon Z\times_V U\to Z\times_Y X$ is clearly a monomorphism. We are now going to show that it is also a regular epimorphism, i.e. that $\varphi x^{\circ}=y^{\circ}f$. This will imply that the square  {\footnotesize\fbox{1}} is a pullback.

Consider the effective equivalence relations $R=R_{ux}, S=R_{\varphi}$ and $T=R_x$ (to use the same notations as in Theorem~\ref{Bourn}). The fact that the whole rectangle {\footnotesize\fbox{1}\fbox{2}} is a pullback, implies that $R_{ux} \wedge R_{\varphi}=1$ and $R_{ux}R_{\varphi}=R_{\varphi}R_{ux}$. Since $R_{ux}\wedge R_{\varphi}=1\leqslant R_x\leqslant R_{ux}$, we conclude that $R_xR_{\varphi}=R_{\varphi}R_x$ by Theorem~\ref{Bourn}. Equivalently, one has the equality $x^{\circ}x \varphi^{\circ}\varphi = \varphi^{\circ}\varphi x^{\circ}x$.

Since the vertical morphisms in {\footnotesize\fbox{1}} are split epimorphisms, then the comparison morphism $R_x\to R_y$ is also a split epimorphism, thus the image of $R_x$ along $\varphi$ is $R_y$: $\varphi(R_x)=R_y$ or, equivalently, $\varphi x^{\circ}x \varphi^{\circ} = y^{\circ} y$. One then has the following equalities:
$$\begin{array}{rcl}
  \varphi x^{\circ} & = & \varphi \varphi^{\circ}\varphi x^{\circ}x x^{\circ} \\
  & = & \varphi x^{\circ}x \varphi^{\circ}\varphi x^{\circ} \\
  & = &  y^{\circ} y \varphi x^{\circ} \\
  & = & y^{\circ} f x x^{\circ} \\
  & = & y^{\circ}f.
\end{array}
$$

Consequently, the arrow $\langle \varphi, x \rangle\colon Z\times_V U\to Z\times_Y X$ is an isomorphism, and the square {\footnotesize\fbox{1}} is a pullback. The square {\footnotesize\fbox{2}} is then also a pullback, since the change-of-base functor along a regular epimorphism in a regular category reflects isomorphisms (see Proposition 2.7 of~\cite{JK}, for instance).
\end{proof}

From the proposition above we shall deduce a general result in the context of almost exact Gumm categories, which extends Lemma 1.1 in ~\cite{Gran}. Recall that a regular category $\CC$ is called an \emph{almost exact category} \cite{JS} when any regular epimorphism in $\mathbb C$ is an effective descent morphism.
This property can be expressed as follows \cite{JST}: for any regular epimorphism $f\colon \xymatrix@=15pt{ X \ar@{>>}[r] & Y}$ and any vertical discrete fibration $(g,h) \colon R \rightarrow R_f$ in $\CC$
$$
\xymatrix@R=25pt@C=30pt{
  R \ar@<4pt>[r]^-{r_1} \ar@<-4pt>[r]_-{r_2} \ar[d]_-{h} \ar@{}[dr]|(.35){\lrcorner} &  X' \ar[l] \ar[d]^-{g} \\
  R_f \ar@<4pt>[r]^-{f_1} \ar@<-4pt>[r]_-{f_2} & X \ar[l] \ar@{>>}[r]^-{f} & Y}
$$
with $R$ an equivalence relation on $X'$, then $R$ is an effective equivalence relation.
\begin{proposition}
\label{efficiently regular Gumm}
Let $\CC$ be an almost exact Gumm category. Consider a commutative diagram
\begin{equation}
\label{|1|2|}
  \vcenter{\xymatrix@!0@R=45pt@C=60pt{
    Z\times_V U \ar@{>>}[d]_-{\varphi} \ar@{>>}[r]^-x \ar@{}[dr]|(.7){\ulcorner} \ar@{}[dr]|-{1} &
    X \ar@{>>}[d]_-f \ar[r]^-u \ar@{}[dr]|-{2} & U \ar[d]^-{w}  \\
    Z \ar@{>>}[r]_-y & Y \ar[r]_-v & V, }}
\end{equation}
such that the whole rectangle is a pullback and the left square \emph{\footnotesize\fbox{1}} is a pushout of regular epimorphisms.
Then both squares \emph{\footnotesize\fbox{1}} and \emph{\footnotesize\fbox{2}} are pullbacks.
\end{proposition}
\begin{proof} 

Consider the following commutative diagram where {\footnotesize\fbox{3}} gives the image factorisation of $\langle x\varphi_1, x\varphi_2 \rangle\colon R_{\varphi}\to X\times X$, and the morphism $\rho$ is induced by the equality $wur_1=wur_2$:
$$
  \xymatrix@!0@R=45pt@C=65pt{
    R_{\varphi} \ar@<-4pt>[d]_-{\varphi_1} \ar@<4pt>[d]^-{\varphi_2} \ar@{>>}[r]^-{p} \ar@{}[dr]|-{3} &
    R \ar@<-4pt>[d]_-{r_1} \ar@<4pt>[d]^-{r_2} \ar[r]^-{\rho} \ar@{}[dr]|-{4} & R_{w} \ar@<-4pt>[d]_-{w_1} \ar@<4pt>[d]^-{w_2} \\
    Z\times_V U \ar@{>>}[d]_-{\varphi} \ar[u] \ar@{>>}[r]_-x \ar@{}[dr]|-{1} & X \ar@{>>}[d]^-f \ar[r]_-u \ar[u] \ar@{}[dr]|-{2} & U \ar[d]^-{w} \ar[u]\\
    Z \ar@{>>}[r]_-y & Y \ar[r]_-v & V.}
$$
Since {\footnotesize\fbox{1}\fbox{2}} is a pullback, then (any of the commutative squares) {\footnotesize\fbox{3}\fbox{4}} is also a pullback. We can apply Proposition~\ref{regular Gumm} to {\footnotesize\fbox{3}\fbox{4}} to conclude that (any of the commutative squares) {\footnotesize\fbox{3}} and {\footnotesize\fbox{4}} are pullbacks.

Note that $R=x(R_{\varphi})$ is an equivalence relation: it is necessarily reflexive and symmetric being the image of the equivalence relation $R_{\varphi}$ along a regular epimorphism $x$. It is also transitive: indeed, as in the proof of Proposition~\ref{regular Gumm}, the assumptions still guarantee that $R_{\varphi}R_x=R_xR_{\varphi}$, and this implies that
$$
\begin{array}{rcl}
  RR & = & x\varphi^{\circ}\varphi x^{\circ}x\varphi^{\circ}\varphi x^{\circ} \\
  & = & xx^{\circ}x\varphi^{\circ}\varphi\varphi^{\circ}\varphi x^{\circ} \\
  & = & x\varphi^{\circ}\varphi x^{\circ} \\
  & = & R.
\end{array}
$$
Since regular epimorphisms are effective for descent in $\CC$, the equivalence relation $R$ is effective, and it is then the kernel pair of its coequaliser. Moreover, the fact that the square {\footnotesize\fbox{1}} is a pushout easily implies that this coequaliser is $f\colon \xymatrix@=15pt{ X \ar@{>>}[r] & Y}$, and $R = R_f = f^{\circ} f$. To complete the proof, one applies the Barr-Kock Theorem~\cite{EC} twice to conclude that {\footnotesize\fbox{1}} and {\footnotesize\fbox{2}} are pullbacks.
\end{proof}

\begin{remark}
\emph{Any efficiently regular category in the sense of \cite{BournBaersums} is almost-exact, so that the result above is true in particular in efficiently regular Gumm categories. For instance, this is the case for any category of topolo-gical Mal'tsev algebras \cite{JP}, then in particular for the category of topological groups. Also any almost abelian category in the sense of \cite{Rump} is almost exact (see \cite{GN}). Another example of almost exact category is provided by the category of regular epimorphisms in an exact Goursat category \cite{JS}.}
\end{remark}

\section{An application to Galois Theory}\label{Galois}
In this section we give an application of the results from Section~\ref{Pullback properties} in Categorical Galois Theory~\cite{JanGalois,JK}.

A commutative square of regular epimorphism
$$
  \xymatrix@1{ P \ar@{>>}[r]^-x \ar@{>>}[d]_-a & X \ar@{>>}[d]^-u \\ A \ar@{>>}[r] & U}
$$
is called \emph{right saturated}~\cite{GJRU} when the comparison morphism $\bar{x}\colon R_a \to R_u$ is also a regular epimorphism.

\begin{proposition} Let $\CC$ be a regular Gumm category. Consider a commutative cube
\label{cube pb property}
\begin{equation}
\label{cube}
  \vcenter{\xymatrix@1@R=30pt@C=40pt{
    P \ar@{>>}[rr]^-x \ar@<-2pt>[dd]_-{\varphi} \ar@{>>}[dr]^-a & & X \ar@<-2pt>@{-->}[dd]_(.3)f \ar@{>>}[dr]^-u \\
    & A \ar@<-2pt>[dd]_(.3)g \ar@{>>}[rr] & & U \ar@<-2pt>[dd]_-w \\
    Z \ar@<-2pt>[uu]_-{\sigma} \ar@{-->>}[rr]_(.7){y} \ar@{>>}[dr]_-b & & Y \ar@<-2pt>@{-->}[uu]_(.7)s \ar@{-->>}[dr]_-v  \\
    & B \ar@<-2pt>[uu]_(.7)t \ar@{>>}[rr] & & V \ar@<-2pt>[uu]_-i}}
\end{equation}
of vertical split epimorphisms and regular epimorphisms. If the left and back faces are pullbacks and the top and bottom faces are right saturated, then the front and right faces are also pullbacks.
\end{proposition}
\begin{proof}
We take the kernel pairs of $a, b, u$ and $v$ and the induced morphisms between them:
$$
  \xymatrix@1@R=30pt@C=40pt{
    R_a \ar@{>>}[rr]^-{\bar{x}} \ar@<-2pt>[dd]_-{\bar{\varphi}} \ar@<-2pt>[dr]_-{a_1} \ar@<2pt>[dr]^-{a_2} & &
    R_u \ar@<-2pt>@{-->}[dd]_(.3){\bar{f}} \ar@<-2pt>[dr]_-{u_1} \ar@<2pt>[dr]^-{u_2} \\
    & P \ar@<-2pt>[dd]_(.3){\varphi} \ar@{>>}[rr]^(.3)x & & X \ar@<-2pt>[dd]_-f \\
    R_b \ar@<-2pt>[uu]_-{\bar{\sigma}} \ar@{-->>}[rr]_(.7){\bar{y}} \ar@<-2pt>[dr]_-{b_1} \ar@<2pt>[dr]^-{b_2} & &
    R_v \ar@<-2pt>@{-->}[uu]_(.7){\bar{s}} \ar@<-2pt>@{-->}[dr]_-{v_1} \ar@<2pt>@{-->}[dr]^-{v_2}  \\
    & Z \ar@<-2pt>[uu]_(.7){\sigma} \ar@{>>}[rr]_-y & & Y. \ar@<-2pt>[uu]_-s}
$$
Note that $\bar{x}$ and $\bar{y}$ are regular epimorphisms since the top and bottom faces of the cube (\ref{cube}) are right saturated. The left and front faces above are pullbacks, so that the rectangle formed by the back and right faces is a pullback. We can apply Proposition~\ref{regular Gumm} to conclude that both the back and right faces above are pullbacks. By the Barr-Kock Theorem the right face of diagram (\ref{cube}) is then a pullback, hence so is the front face of (\ref{cube}).
\end{proof}

As a consequence of Proposition \ref{cube pb property} we give a new proof of Theorem 4.8 of~\cite{JK} stating that every central and split extension is a trivial extension for the Galois structure associated with any Birkhoff subcategory of an exact Goursat category. Let us briefly recall the main definitions, and we refer to ~\cite{JK} for more details.

When $\CC$ is an exact category and $\XX$ a full replete subcategory of $\CC$
\begin{equation}
\label{adjunction}
  \xymatrix@1{ \CC \ar@<6pt>[r]^-I \ar@{}[r]|-{\bot} & \;\XX, \ar@<6pt>@{_{(}->}[l]}
\end{equation}
one calls $\XX$ a \emph{Birkhoff} subcategory of $\CC$ when $\XX$ is stable in $\CC$ under subobjects and regular quotients. Equivalently, all $X$-components $\eta_X$ of the unit $\eta\colon 1_{\CC}\Rightarrow I$ of the adjunction  are regular epimorphisms (the right adjoint is assumed to be a full inclusion and will not be mentioned explicitely), and the naturality square
\begin{equation}
\label{nat square}
  \vcenter{\xymatrix{X \ar@{>>}[r]^-{\eta_X} \ar@{>>}[d]_-f & IX \ar@{>>}[d]^-{If} \\
  Y \ar@{>>}[r]_-{\eta_Y} & IY}}
\end{equation}
is a pushout for any regular epimorphism $f\colon \xymatrix@=15pt{ X \ar@{>>}[r] & Y}$.

A regular epimorphism $f\colon \xymatrix@=15pt{ X \ar@{>>}[r] & Y}$ is called a \emph{trivial extension} when the naturality square (\ref{nat square}) is a pullback. It is called a \emph{central extension} when it is ``locally'' trivial: there exists a regular epimorphism $y\colon \xymatrix@=15pt{ Z \ar@{>>}[r] & Y}$ such that the pullback of $f$ along $y$ is a trivial extension.

\begin{theorem}\emph{(Theorem 4.8 of~\cite{JK})}
\label{central split}
Let $\CC$ be an exact Goursat category, and $\XX$ a Birkhoff subcategory of $\CC$. Then every central and split extension is necessarily a trivial extension.
\end{theorem}
\begin{proof} Let $f\colon \xymatrix@=15pt{ X \ar@{>>}[r] & Y}$ be both a central extension and a split epimorphism. By definition, there exists a regular epimorphism $y\colon \xymatrix@=15pt{ Z \ar@{>>}[r] & Y}$ such that the pullback of $f$ along $y$ is a trivial extension. So, in the following commutative cube
$$
  \xymatrix@1@R=30pt@C=40pt{
    P \ar@{>>}[rr]^-x \ar@<-2pt>[dd]_-{\varphi} \ar@{>>}[dr]^-{\eta_P} & & X \ar@<-2pt>@{-->}[dd]_(.3)f \ar@{>>}[dr]^-{\eta_X} \\
    & IP \ar@<-2pt>[dd]_(.3){I\varphi} \ar@{>>}[rr] & & IX \ar@<-2pt>[dd]_-{If} \\
    Z \ar@<-2pt>[uu]_-{\sigma} \ar@{-->>}[rr]_(.7){y} \ar@{>>}[dr]_-{\eta_Z} & & Y \ar@<-2pt>@{-->}[uu]_(.7)s \ar@{-->>}[dr]_-{\eta_Y}  \\
    & IZ \ar@<-2pt>[uu]_(.7){I\sigma} \ar@{>>}[rr] & & IY, \ar@<-2pt>[uu]_-{Is}}
$$
the back face is a pullback by construction, and the left face is a pullback by the assumption that $\varphi$ is a trivial extension. Note that the top and bottom faces are pushouts of regular epimorphisms and are then right saturated (by Proposition 7.1 of~\cite{CKP}). By Proposition~\ref{cube pb property} we conclude that the front and right faces are pullbacks. It follows that $f$ is a trivial extension.
\end{proof}
A regular epimorphism $f\colon \xymatrix@=15pt{ X \ar@{>>}[r] & Y}$ is called a \emph{normal extension} if the pullback of $f$ along itself is a trivial extension. By definition any normal extension is central, but the converse is false in general, as various counter-examples given in \cite{JK} show. In the exact Goursat context the notions of central and normal extensions coincide:
\begin{corollary}\label{Coincidence}\cite{JK}
Consider a Birkhoff adjunction \emph{(\ref{adjunction})} where $\CC$ is an exact Goursat category. Then every central extension is normal.
\end{corollary}
\begin{proof}
This follows immediately from Theorem \ref{central split} and the fact that central extensions are pullback stable (this follows from the fact that adjunction is ``admissible'' in the sense of Categorical Galois Theory, see the Remark \ref{final} here below).
\end{proof}

\begin{remark}\label{final}
\emph{A Birkhoff subcategory $\XX$ of an exact category $\CC$ is called \emph{admissible} when $I$ preserves pullbacks of the form
\begin{equation}\label{admissibilitysquare}
  \vcenter{\xymatrix{A \ar[r]^-{n} \ar@{>>}[d]_-{\varphi} \ar@{}[dr]|(.35){\lrcorner} & U \ar@{>>}[d]^-{w} \\
  B \ar[r]_-{m} & V,}}
\end{equation}
where $w\colon \xymatrix@=15pt{ U \ar@{>>}[r] & V}$ is a regular epimorphism of $\XX$ (see Proposition 3.3 of~\cite{JK}). In~\cite{Rossi} V. Rossi proved that any Birkhoff subcategory of an almost exact Gumm category is admissible, extending a result due to G. Janelidze and G.M. Kelly \cite{JK}. The proof of the more general Proposition \ref{efficiently regular Gumm} above is actually similar to the one given in~\cite{Rossi}. In order to deduce the admissibility result from Proposition \ref{efficiently regular Gumm} it suffices to decompose the square \eqref{admissibilitysquare} above as
$$
  \vcenter{\xymatrix{A \ar@{>>}[r]^-{\eta_A} \ar@{>>}[d]_-{\varphi} & IA \ar[r]^-{In} \ar@{>>}[d]^-{I\varphi} & U \ar@{>>}[d]^-{w} \\
  B \ar@{>>}[r]_-{\eta_B} & IB \ar[r]_-{Im} & V}}
$$
to conclude that both squares are pullbacks.}
\end{remark}


\end{document}